\documentclass[11pt]{article}
\usepackage[T1]{fontenc}
\usepackage[utf8]{inputenc}
\usepackage{authblk}
\RequirePackage{amsthm,amsmath, mathrsfs, amssymb, easybmat, graphicx}
\RequirePackage[numbers]{natbib}
\RequirePackage[colorlinks,citecolor=blue,urlcolor=blue]{hyperref}
\usepackage[left = 1.2in, right = 1.2in, top = 1.3in, bottom = 1.3in]{geometry}
\usepackage{psfrag}
\usepackage{subfigure}

\numberwithin{equation}{section}
\theoremstyle{plain}
\newtheorem{Theorem}{Theorem}[section]
\newtheorem{Remark}{Remark}[section]
\newtheorem{Lemma}{Lemma}[section]

\newtheorem{Proposition}{Proposition}[section]
\newtheorem{Example}{Example}[section]
\newtheorem{Corollary}{Corollary}[section]

\newcommand{\R}{\mathbb{R}}
\newcommand{\bx}{\mathbf x}
\newcommand{\N}{\mathcal{N}}
\newcommand{\PD}[1]{\rm{S}_{++}(#1)}

\newcommand{\match}[1]{\overset{{#1}}{=}}
  \newcommand{\tr}{\textrm{tr}}                      
\newcommand{\rk}{\textrm{rank}}
\newcommand{\G}{\mathcal{G}}
\newcommand{\vb}{\mathbf{v}}

\def\u{\mathbf{u}}

\def\w{\mathbf{w}}
\def\x{\mathbf{x}}
\def\y{\mathbf{y}}
\def\ra{\mathrm{r}_{w}}
\def\rb{\mathrm{r}}

\newcommand{\Hc}{\mathcal{H}}
\def\Sc{\mathcal{S}}

\title{Sharper lower and upper bounds for the Gaussian rank of a graph}
\author{Emanuel Ben-David\thanks{ehb2126@columbia.edu}}
\affil{Department of Statistics, Columbia University}
\date{}

\begin{document}
\maketitle

\begin{abstract}

An open problem in graphical Gaussian models is to determine the smallest number of observations needed to guarantee the existence of the maximum likelihood estimator of the covariance matrix with probability one. In this paper we formalize a closely related problem in which the existence of the maximum likelihood estimator is guaranteed for all generic observations. We call the number determined by this problem the Gaussian rank of the graph representing the model. We prove that the Gaussian rank is strictly between the subgraph connectivity number and the graph degeneracy number. These bounds are in general much sharper than the best bounds known in the literature and furthermore computable in polynomial time.
\end{abstract}


\section{ Introduction}\label{sec:intro}
An open problem in graphical Gaussian models is to determine the smallest number of observations needed to guarantee the existence of the maximum likelihood estimator (MLE) of the covariance matrix. This problem first arose in the Dempster's paper \cite{D72} and has been frequently brought to attention by Steffen Laurtizen, as evident in \cite{Buhl93} and Lauritzen's lectures in ``Durham Symposium on Mathematical Aspects of Graphical Models 2008." Hence we refer to this problem as the D\&L problem using the initials of Dempster and Lauritzen. To well pose this problem it is necessary to exclude some observations from a set of at least zero-probability. For this purpose, we can formalize the D\&L problem in its most general form as follows. \\

\noindent \textbf{D\&L problem (I)} For a given graphical Gaussian model with respect to a graph $\G$, determine the smallest number of observations needed to guarantee the existence of the MLE of the covariance matrix with probability one.\\

A slightly different well-posed formalization of the D\&L problem is to require the existence of the MLE (of the covariance matrix, henceforth) for all generic observations in the following sense. The observations\footnote{Throughout the paper, we always assume that the observations are independent and identically distributed.} $\x_1,\ldots, \x_n$  are said to be generic if each $n\times n$ principal submatrix of the sample covariance matrix $\mathcal{S}=1/n\sum_{i=1}^n\x_i\x_i^{\top}$ is non-singular. Note that with probability one every $n$ observations are generic (thus the set of non-generic observations has zero probability). The problem is now formalized as follows.\\

\noindent \textbf{D\&L problem (II)} For a given graphical Gaussian model with respect to a graph $\G$, determine the smallest number of observations needed to guarantee the existence of the MLE for all generic observations.\\

The number determined by the D\&L problem (II) is said to be the Gaussian rank of $\G$, denoted by $\rb(\G)$. The primary goal of this paper is to obtain sharp lower and upper bounds on $\rb(\G)$. In parallel to the D\&L problem (II), the number determined by the D\&L problem (I) is said to be the weak Gaussian rank of $\G$, denoted by $\ra(\G)$. Note again that being in general position is a generic property of any $n$ observations, that is, with probability one for every $n$ observations each $n\times n$ principal submatrix of the sample covariance matrix is in non-singular. This obviously implies that $\rb(\G)$ is an upper bound for $\ra(\G)$.  

As we mentioned earlier, the D\&L problem first arose in \cite{D72}, a celebrated work of Dempster in which, under the principle of parsimony in fitting parametric models, Dempster introduced graphical Gaussian models and gave a partial analytic expression for the MLE of the covariance matrix, assuming that the MLE exists. In practice, the MLE has to be computed by an iterative numerical procedure, such as the iterative proportional scaling procedure (IPS) \cite{S86}. However, if the number of observations is not sufficiently large there is no guarantee that the output matrix will indeed correspond to the MLE, as it might not be positive definite.\\
Both formalizations (I) and (II) of the D\&L problem are important because, in some sense, their answers give us a better understanding of sparse graphs, even when the main concern is not the MLE. Moreover, knowing exactly, or closely enough, how many observations are needed to estimate the covariance matrix is useful in applications of graphical Gaussian models in situations where the sample size calculation is a very important aspect of the study, for example in genetics and medical imaging.

In a graphical Gaussian model the sparsity is given by a pattern of zeros in the inverse covariance matrix. The pattern of zeros is determined by the missing edges of a graph and each zero entry of a generic inverse covariance matrix indicates that the corresponding pair of variables are conditionally independent given the rest of variables.  An attractive feature of graphical Gaussian models, or graphical models in general, is that if the graph representing the model is sparse, then the MLE of the covariance matrix can exist even if the number of observations is much smaller than the dimension of the multivariate Gaussian distribution.  Intuitively, we expect that the guassian rank of the graph, the number determined by the D\&L problem, to decrease as the graph becomes sparser. However, it is not clear what best measures the sparsity of a graph or how the Gaussian rank varies accordingly. The main theorem of this paper suggests two such measures of sparsity.\\

Despite some efforts since 90's not much progress has been made to resolve the D\&L problem. The existing results are limited to a handful of publications as follows. 
\begin{itemize}
\item[(R1)] \cite{G84}  For a decomposable graph $\G$, a graph that has no induced cycle of length larger than or equal to four, the Gaussian rank is equal to $\omega(\G)$, the size of the largest complete subgraph of $\G$.
\item[(R2)] \cite{G84} For every graph $\G$, $\omega(\G)\le \ra(\G)=\rb(\G)\le \textrm{tw}(\G)+1$, where $\textrm{tw}(\G)$ denotes the treewidth of $\G$ (please see Appendix A for the definition). 
\item[(R3)] \cite{Buhl93} For $C_p$, a cycle of length $p\ge 3$, $2=\omega(C_p)<\ra(C_p)=\rb(C_p)=3=\rm{tw}(C_p)+1$.
\item[(R4)] \cite{Uh12} For $G_{3,3}$, the $3\times 3$ grid, $2=\omega(G_{3,3})<\ra(G_{3,3})=\rb(G_{3,3})=3<\rm{tw}(G_{3,3})+1=4$.
\end{itemize}
\noindent In the literature the bounds given by (R2) are currently the best known bounds for the guassian rank. Restricted to the class of decomposable graphs these bounds are tight, since $\textrm{tw}(\G)+1=\omega(\G)$ for a decomposable graph $\G$, but we may note that (R4) in \cite{Uh12} shows that for non-decomposable graphs the bounds in (R2) are not necessarily tight. Intuitively, it is apparent that $\omega(\G)$ overestimates and $\textrm{tw}(\G)$ underestimates the sparsity of $\G$ and therefore sharper bounds may exist. In fact in this paper we give much sharper bounds on the Gaussian rank. The lower and upper bounds we give are the subgraph connectivity number, denoted by $\kappa^{*}(\G)$, and the graph degeneracy number, denoted by $\delta^{*}(\G)$. Both these bounds are well-known in graph theory. Formally we prove the following theorem.\\
\begin{Theorem}\label{thm:main} Let $\G=(V, E)$ be a graph. Then
\begin{equation}\label{eq:main}
\kappa^{*}(\G)+1\le \rb(\G)\le \delta^{*}(\G)+1
\end{equation}
\end{Theorem}
\noindent  All the results stated in (R1) through (R4) now immediately follow from Theorem \ref{thm:main}, since by some simple calculations we can show that
\begin{enumerate}
	\item[(a)] for a decomposable graph $\G$, $\omega(\G)=\kappa^{*}(\G)+1=\delta^{*}(\G)+1=\rm{tw}(\G)+1$;
	\item [(b)] for any (arbitrary) graph $\G$,  $\omega(\G)\le \kappa^{*}(\G)+1\le \rb(\G)\le \delta^{*}(\G)+1\le \textrm{tw}(\G)+1$;
	\item[(c)] for a cycle (of any length) $\G$, $\kappa^{*}(\G)=\delta^{*}(\G)=2$;
	\item [(d)] for a $k\times m$ grid (with $k$ and $m\ge 2$), $\kappa^{*}(\G)=\delta^{*}(\G)=2$.
\end{enumerate}
Note that by Part (d) the guassian rank of every grid is $3$ which is substantially less than the upper bound given by (R2) for grids of large dimensions. The reason is that the treewidth of a $k\times m$ grid is $\min\{k,m\}$ which tends to $+\infty$ as $k$ and $ m\to +\infty$ \cite{K94}.\\

The organization of the paper is as follows. In section \ref{sec:pre} we review some basic notation and terminology in graphical models and matrix algebra, and explain how the maximum likelihood problem for graphical guassian models leads to the D\&L problem.  In this section we also discuss the concept of in general positions for vectors and matrices, which is a fundamental concept in subsequent sections. In section \ref{sec:grank} we give an alternative description of the Gaussian rank which we then use to derive some of the properties of the Gaussian rank. In section \ref{sec:main} we prove Theorem \ref{thm:main}. The proof of the upper bound is based on two key observations that if a vertex $v$ is removed from a graph $\G$, then (1) the Gaussian rank of the resulting graph is at most one less than $\rb(\G)$  (2) the Gaussian rank of the resulting graph remains equal to $\rb(\G)$ if it is larger than the number of the vertices adjacent to $v$. The proof of the lower bound is based on the concept of orthogonal representations of graphs and a theorem in \cite{L89}. In section  \ref{sec:app} we apply Theorem \ref{thm:main} to some special graphs to exactly determine their Gaussian ranks. These graphs include symmetric graphs and random graphs. We also obtain a tight numerical upper bound for the Gaussian ranks of planar graphs. 

\section{Preliminaries}\label{sec:pre}
In this section, we establish some necessary notation, terminology and definitions in graphical models, matrix algebra and geometric graph theory. We also carefully explain how the maximum likelihood problem for graphical guassian models leads to the D\&L problem.
\subsection{Graph theoretical notion, definitions} Our notation presented here closely follows the notation established in \cite{M78} and \cite{L96}. Let $\G=(V, E)$ be an undirected graph, where $V=V(\G)$ is the vertex set and $E=E(\G)$ is the edge set of $\G$. Each edge $e\in E$ is an unordered pair $ij= \{i,j\}$, where $i,j\in V$. For ease of notation, as in \cite{Uh12}, we assume that $E$ contains all the self-loops, that is,  $(i,i)$ for every $i\in V$. Otherwise stated, we always assume that $V(\G)=\{1,\ldots,p\}= [p]$.
\begin{enumerate}
\item[(1)] For a vertex $v\in V$, the set and the number of vertices adjacent to $v$ are denoted by $\mathrm{ne}(v)$ and $\deg_{\G}(v)$.
\item[(2)] A graph $\mathcal{H}$ is a subgraph of $\G$, denoted by $\mathcal{H}\subseteq \G$, if $V(\mathcal{H})\subseteq V(\G)$ and $E(\mathcal{H})\subseteq E(\G)$.
\item[(3)] For a set $V'\subseteq V$, the graph defined by $G[V']= (V', E\cap V'\times V')$ is said to be the subgraph of $\G$ induced by $V'$.
\item[(4)] For a vertex $v\in V$, the induced subgraph $\G[V\setminus\{v\}]$ is denoted by $\G-v$. Note that $\G-v$ is simply obtained by removing  $v$ and its incident edges from $\G$.
\item[(5)] For an edge $e\in E$, $\G-e$ denotes the subgraph obtained by removing $e$ from $\G$, that is, $\G-e=(V, E\setminus\{e\})$.
\end{enumerate}

\subsection{Graph parameters} Let $\mathscr{G}$ denote the set of all graphs. A graph parameter is a function $\left(\mu: \mathscr{G}\to \mathbb{R}\right):\G \mapsto \mu(\G)$. In words, a graph parameter is a function that assigns a number to each graph. In graph theory two common graph parameters are:\\
\noindent(a)~ $\delta(\G)= \min\{ \deg_{\G}(v):\: v\in V\}$, said to be the minimum degree of $\G$;\\
\noindent(b)~ $\kappa(\G)= \min\{|S|: \:  S\subseteq V \:\; \text{such that} \:\; \G[V\setminus S]\: \:\text{is disconnected} \}$, said to be the vertex connectivity number of $\G$. In words, $\kappa(\G)$ is the smallest number $k$ of vertices whose deletion separates the graph or makes it trivial.\\
\noindent Now for every graph parameter $\mu$ we can define a new graph parameter as
\[
\mu^*(\G)=\max\left\{\mu(\mathcal{H}): \: \mathcal{H}\subseteq \G\right\}.
\]
Two such defined graph parameters are $\delta^{*}(\G)$ and $\kappa^{*}(\G)$, known as the graph degeneracy number and the subgraph connectivity number of $\G$, respectively \cite{M78}. Note that $\kappa^*(\G)$  is the smallest number of vertex deletions sufficient to disconnect each subgraph of $\G$. Since $\kappa(\G)\le \delta(\G)$ we have
\begin{equation}\label{eq:kappa}
\kappa^{*}(\G)\le \delta^{*}(\G) \: \; \text{(see Part (a) in Remark \ref{rem:star}).}
\end{equation}
\begin{figure}[htbp]
\centering
\subfigure[]{
\includegraphics[scale=.08]{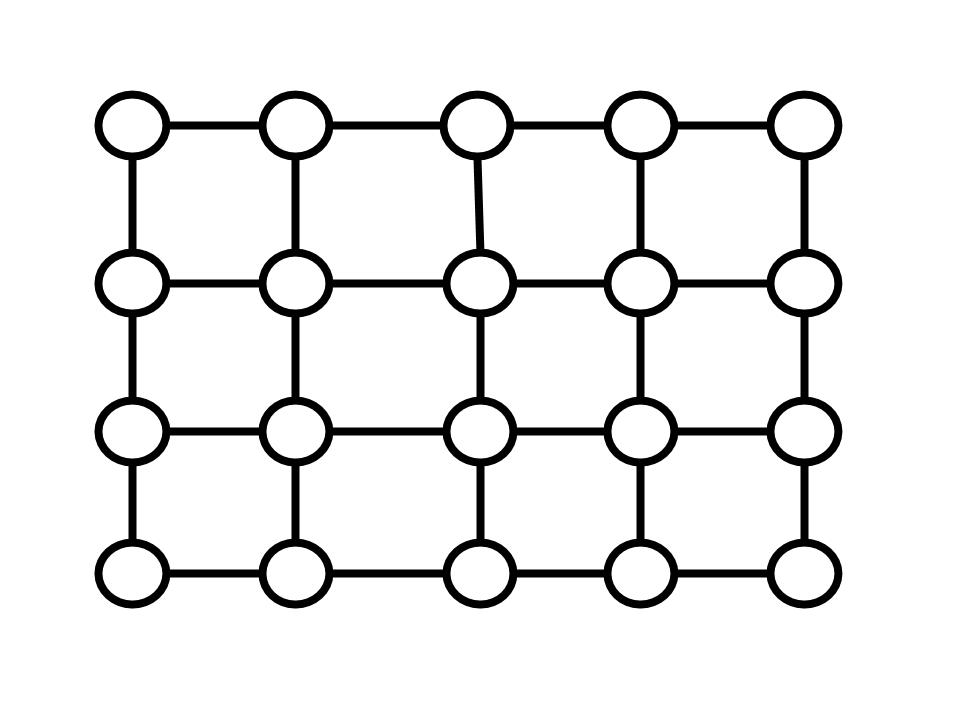}
\label{fig:subfig-a}
}
\hspace{1cm}
\subfigure[]{
\includegraphics[scale=.07]{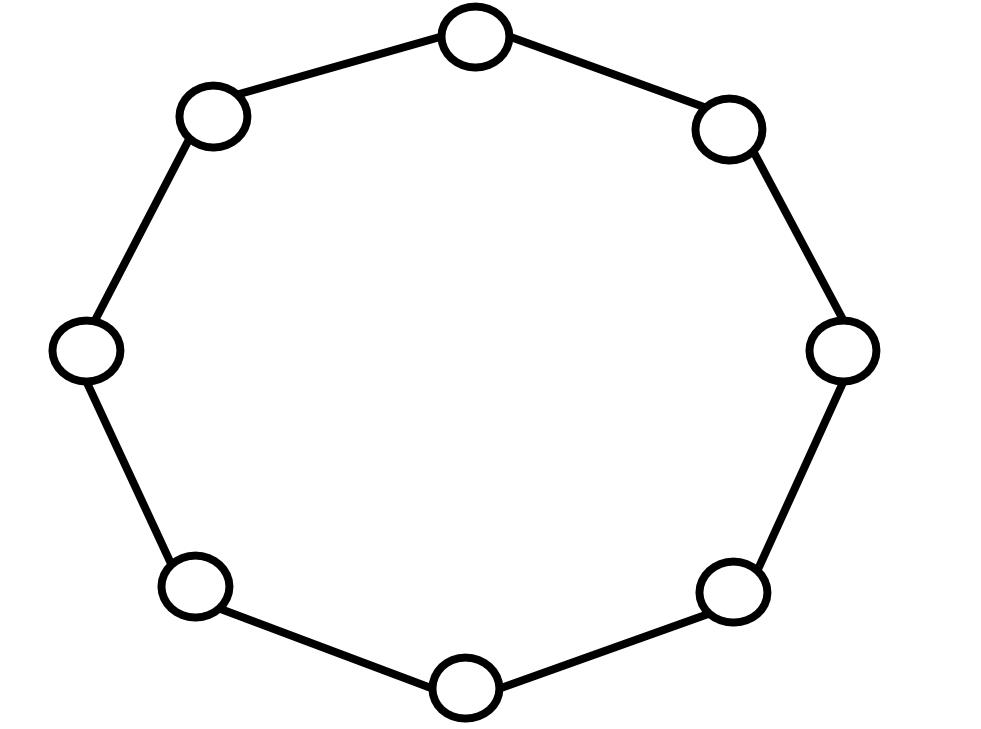}
\label{fig:subfig-b}
}
\hspace{1cm}
\subfigure[]{
\includegraphics[scale=.08]{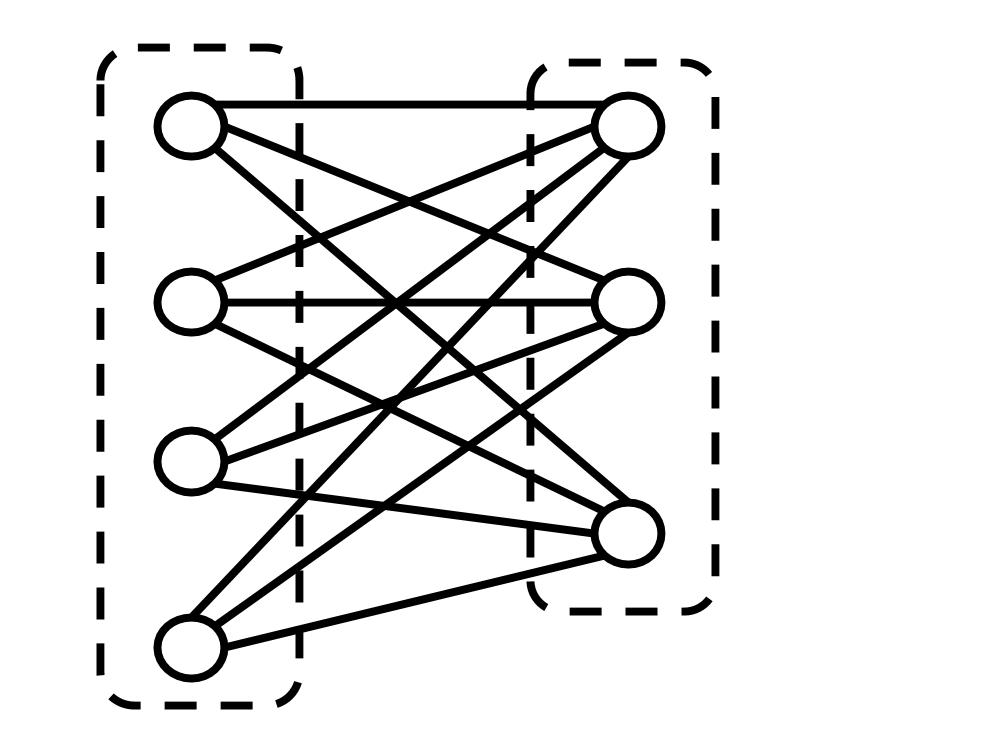}
\label{fig:subfig-c}
}
\caption{Denoted graphs are \rm{(a)} a $4\times 5$ grid, $G_{4, 5}$ \rm{(b)} a cycle of length 8, $C_8$ \rm{(c)} a complete $(4,3)$-bipartite graph, $K_{4,3}$.}
\end{figure}
\begin{Example}\label{ex:1}
Consider the graphs given by Figure \ref{fig:subfig-a}, \ref{fig:subfig-b} and \ref{fig:subfig-c}.
\begin{enumerate}
	\item[\rm{(a)}] The graph  denoted by $G_{4,5}$ in Figure \ref{fig:subfig-a} is a $4\times 5$ grid. It is easy to see that $\kappa^*(G_{4,5})=\delta^*(G_{4,5})=2$. In fact for every $k\times m$ grid (with $k$ and $m\ge 2$), $\kappa^*(G_{k, m})=\delta^*(G_{k,m})=2$.
	\item[\rm{(b)}] The graph denoted by $C_8$ in Figure \ref{fig:subfig-b} is a cycle of length $8$ with $\kappa^*(C_8)=\delta^*(C_8)=2$. This is obviously true for any arbitrary cycle $C_p$ (with length $p\ge 3$).
	\item[\rm{(c)}]  The graph  denoted by $K_{4,3}$ in Figure \ref{fig:subfig-c} is a complete $(4,3)$-bipartite graph. In general, a complete $(k,m)$-bipartite graph is a graph so that its vertex set can be partitioned into two sets with $k$ and $m$ elements such that two vertices are adjacent if and only if they are in different partitions of $V$. One can check that $\kappa^*(K_{k,m})=\delta^*(K_{k,m})=\min\{k,m\}$ \cite{M78}.
\end{enumerate}
\end{Example}
\begin{Remark}\label{rem:star}
These simple facts can be found in \cite{M78}.\\
\noindent \rm{(a)}~ If two graph parameters are such that $\mu_1(\mathcal{H})\le \mu_2(\mathcal{H})$ for every subgraph $\mathcal{H}\subseteq \G$, then $\mu_1^{*}(\G)\le \mu_2^{*}(\G)$. \\
\noindent \rm{(b)}~ If a graph parameter is non-decreasing, that is, $\Hc\subseteq \G\implies \mu(\Hc)\le \mu(\G)$, then $\mu^*(\G)=\mu(\G)$. For example, for this reason $\omega^*(\G)=\omega(\G)$ and $\rm{tw}^*(\G)=\rm{tw}(\G)$.\\
\noindent \rm{(c)}~ For every graph $\G$ we have $\omega(\G)\le \kappa(\G)+1\le \delta(\G)+1\le \rm{tw}(\G)+1$. In light of (a) and (b), this implies that
$\omega(\G)\le \kappa^*(\G)+1\le \delta^*(\G)+1\le \rm{tw}(\G)+1$.\\
\noindent (d)~ There are algorithms (specifically given in \cite{M78}) that compute $\kappa^*(\G)$ and  $\delta^*(\G)$, respectively, in no more than $\rm{O}(|V|^{3/2}|E|^2)$ and $\rm{O}(|E|)$ steps (in contrast, clique number and treewidth cannot be computed in polynomial time).
\end{Remark}


\subsection{Matrix algebra notation, definitions} The notation presented here closely follows the notation established in \cite{Z5} and \cite{M78}. For a vector $\u=(u_i: \: i\in V)\in \R^V$ and $\alpha\subseteq V$, let $\u[\alpha]$ denote the subvector $(u_i:\: i\in \alpha)\in \R^{\alpha}$. For a matrix $A=(A_{i,j})\in \R^{V\times V }$ and $\alpha$ and $\beta\subseteq V$,  let $A[\alpha,\beta]$ denote the submatrix $(a_{i,j}: i\in \alpha, j\in \beta)\in \R^{\alpha\times \beta}$. A principal submatrix $A[\alpha,\alpha]$ is simply denoted by $A[\alpha]$. Some other notations used throughout the paper are as follows. 

\begin{enumerate}
\item[(1)]  The set of $p\times p$  symmetric  matrices is denotes by $S(p)$.
\item[(2)]  The set of (symmetric) positive semi-definite matrices is denoted by $\rm{S}_{+}(p)$.
\item[(3)] The set of (symmetric) positive definite matrices is denoted by  $\rm{S}_{++}(p)$.
\item[(4)]  The set of matrices $A\in \rm{S}_{+}(p)$ of a fixed rank $d$ is denoted by $\rm{S}_{+}(p,d)$.
\item[(5)]  $A\succeq 0$ (or $A\succ 0$) denotes $A$ is a positive semidefinite (or positive definite) matrix without specifying its dimension.
\item[(6)]  The Moore-Penrose generalized inverse of $A$  is denoted by $A^{\dag}$. Note that $A^{\dag}A A^{\dag}=A^{\dag}$ and $A A^{\dag} A=A$. It is clear that $A^{\dag}=A^{-1}$ when $A$ is non-singular.

\item[(7)] Let $A$ be partitioned as $ A=
\begin{pmatrix}
A[\alpha]&A[\alpha,\beta]\\
A[\beta,\alpha]&A[\beta]
\end{pmatrix}
$, where $\beta=V\setminus \alpha$. Then the Schur complement of $A[\alpha]$, denoted by $A[\beta|\alpha]$, is defined by
\[
  A[\beta|\alpha]=A[\beta]-A[\beta,\alpha]A[\alpha]^{\dag}A[\alpha,\beta].
	\]
\end{enumerate}
\textbf{Convention:} For convenience in subsequent sections, otherwise stated, we always assume that given matrices are non-zero.
\subsection{Vectors and matrices in general positions} Let $\vb_1,\ldots,\vb_p$ be $p$ vectors in $\mathbb{R}^{d}$. These vectors are said to be in general position if any $d$ vectors among them are linearly independent. A similar concept is defined for square matrices. A  square matrix $A$ of rank $d$ is said to be in general position if every $d\times d$ principal submatrix of $A$ is non-singular.
\begin{Lemma}\label{lem:generic}
If $p\ge d$ vectors are randomly selected in $\R^d$ (with a probability distribution dominated by Lebesgue measure), then with probability one they are in general position.
\end{Lemma}
\noindent In other words, Lemma \ref{lem:generic} says that being in general position is a generic property of any $p\ge d$ vectors in $\R^d$.
\begin{proof}
First we note that randomly selecting $p$ vectors in $\R^d$ is equivalent to randomly selecting a matrix $W\in\R^{d\times p}$. Let $X=(x_{i,j})$ denote a $d\times p $ matrix of variables. For each $\alpha=\{i_1,\ldots,i_d\}\subseteq [p]$ we define the polynomial
\[
p_{\alpha}(X)=\det\begin{pmatrix}
x_{1,i_1}&\ldots&x_{1,i_d}\\
\vdots&\ddots&\vdots\\
x_{d,i_1}&\ldots&x_{d,i_d}
\end{pmatrix} \: \: \text{and} \: \: \mathbb{V}_{\alpha}=\{W\in\R^{d\times p}: \: p_{\alpha}(W)=0\}.
\]
Each $\mathbb{V}_{\alpha}\subsetneq\R^{d\times p}$ is an algebraic set and its (Lebesgue) measure is therefore zero. Thus,
$
\rm{Pr}\left( W\in R^{d\times p}\setminus \bigcup_{\alpha} \mathbb{V}_{\alpha}\right)=1$.
This completes the proof, since each $d$ columns of a matrix $W\in R^{d\times p}\setminus \bigcup_{\alpha} \mathbb{V}_{\alpha}$ are linearly independent.
\end{proof}
The next lemma shows how two concepts of in general positions, for vectors and square matrices, are related.
\begin{Lemma}\label{lem:equiv}
Let $A\in \rm{S}_{+}(p,d)$. Then $A$ is in general position if and only if  there are in general-position vectors $\vb_1,\ldots,\vb_p\in \R^d$ such that
\begin{equation}\label{eq:gen}
A=\left(
\begin{matrix}
\vb_1^{\top}\\
\vdots\\
\vb_p^{\top}
\end{matrix}
\right)
\left(
\begin{matrix}
\vb_1&\cdots&\vb_p
\end{matrix}
\right)=(\vb_i^{\top}\vb_j)_{1\le i,j\le p}.
\end{equation}
\end{Lemma}
\begin{proof}
\begin{enumerate}
	\item[$\implies)$] Suppose $A\in \rm{S}_{+}(p,d)$. Let  each pair of $(\lambda_1,\u_1),\ldots, (\lambda_d, \u_d)$ denote a positive eigenvalue and its corresponding eigenvector of $A$. Let us set $\w_i=\sqrt{\lambda_i}\u_i$, for $i=1,\ldots,d$. Then we have
	\begin{equation}\label{eq:decomp}
	A=\sum_{i=1}^{d}\w_i\w_i^{\top}=W^{\top}W, \: \text{where}\:\: W=
	\begin{pmatrix}
		\w_1^{\top}\\\vdots\\\w_d^{\top}
	\end{pmatrix}=\begin{pmatrix}
		\vb_1&\ldots&\vb_p
	\end{pmatrix}\in\R^{d\times p}.
	\end{equation}
Note that $\vb_1,\ldots,\vb_p\in \R^d$ are the columns of $W$.  Suppose for $\alpha=\{i_1,\ldots,i_d\}\subseteq [p]$ the vectors $\vb_{i_1},\ldots,\vb_{i_d}$ are linearly dependent. Thus, there is a non-zero vector $\x_{\alpha}=(x_{i_1},\ldots,x_{i_d})^{\top}\in\R^d$ such that
	\[
	\sum_{i\in \alpha}x_{i}\vb_{i}=  W\begin{pmatrix}
		\x_{\alpha}\\0
	\end{pmatrix}=0
	\implies 
	W^{\top}W\begin{pmatrix}
		\x_\alpha\\0
	\end{pmatrix}=A\begin{pmatrix}
		\x_{\alpha}\\0
	\end{pmatrix}=0.
	\]
Thus $A[\alpha]\x_{\alpha}=0$ and consequently $A[\alpha]$ is singular.
	\item[$\Longleftarrow$)]  Suppose $\vb_1,\ldots,\vb_p\in \R^d$ are in general position. Let $A$ be defined as in Equation \eqref{eq:gen}. Then for each $\alpha=\{i_1,\ldots,i_d\}\subseteq [p]$ the submatrix
	\[
	A[\alpha]=\left(
\begin{matrix}
\vb_{i_1}^{\top}\\
\vdots\\
\vb_{i_d}^{\top}
\end{matrix}
\right)
\left(
\begin{matrix}
\vb_{i_1}&\cdots&\vb_{i_d}
\end{matrix}
\right) \; \text{is obviously non-singular.}
	\]
\end{enumerate}
\end{proof}
\begin{Corollary}\label{cor:cor1}
Suppose $\x_1,\ldots,\x_n$ are $n$ observations from a $p$-variate distribution. Then with probability one the sample covariance matrix $\mathcal{S}$ and the sample correlation matrix $\mathcal{R}$ are in general positions.
\end{Corollary}
\begin{proof}
 Let
\[
\mathcal{X}=\begin{pmatrix}\x_1^{\top}\\ \vdots\\\x_n^{\top}\end{pmatrix}=\begin{pmatrix}\y_1& \ldots&\y_p\end{pmatrix}\in\R^{n\times p}.
\]
be the data matrix. Then by Lemma \ref{lem:generic} with probability one $\y_1,\ldots, \y_p$, the columns of $\mathcal{X}$, are in general position in $\R^n$. The result now follows from Lemma \ref{lem:equiv}, since whenever the columns of $\mathcal{X}$ are in general position $\mathcal{S}=1/n\left(\mathcal{X}^{\top}\mathcal{X}\right)$ is also in general position.
\end{proof}
A useful property of an in general-position matrix $A\in \rm{S}_{+}(p,d)$ is that in some sense it can be extended to another in general-position matrix in $\rm{S}_{+}(p+1,r)$ or $\rm{S}_{+}(p+1,r+1)$. The next lemma formalizes this fact.
\begin{Lemma}\label{lem:ext}
Let $A\in \rm{S}_{+}(p,d)$ be in general position.\\\

\noindent (a)~  Then there is a vector $\w\in \R^p$ such that $A-\w\w^{\top}\in \rm{S}_{+}(p,d)$ is in general position. This in return  implies that the matrix
\[
\begin{pmatrix}
	1&\w^T\\
	\w& A
\end{pmatrix}\in \rm{S}_{+}(p+1,d)\: \text{is in general position.}
\]
\noindent (b)~ If $d\le p-1$, then there is a vector $\u\in \R^p$ such that $A+\u\u^{\top}\in \rm{S}_{+}(p,d+1)$ is in general position. This in return  implies that the matrix
\[
\begin{pmatrix}
	1&\u^{\top}\\
	\u& A+\u\u^{\top}
\end{pmatrix}\in \rm{S}_{+}(p+1,d+1)\: \text{is in general position.}
\]
\end{Lemma}
\begin{proof} 
(a)~ Let $\tau=[d]$ and choose a non-zero vector $\x_{\tau}\in \R^d$ such that $\x_{\tau}^{\top}A[\tau]\x_{\tau}<1$. We set
\[
\x=\begin{pmatrix}
	\x_{\tau}\\
	0
\end{pmatrix} \in\R^p \;\: \text{and} \;\:  \w=A\x=
\begin{pmatrix}
	A[\tau]\x_{\tau}\\
	A[\gamma,\tau]\x_{\tau}
\end{pmatrix},
\]
where $\gamma=V\setminus \tau$.
Note that $A\succeq 0$ and $\mathbf{x}^{\top}A\mathbf{x}=\x_{\tau}^{\top}A[\tau]\x_{\tau}<1$. Thus we have
 \[
1-\w^{\top}A^{\dag}\w=1-\x^{\top} A\x>0,
\]
and therefore $B= A-\w\w^{\top}\in\rm{S}_{+}(p)$, by Theorem 1.20 in \cite{Z5}. Now we claim that $B$ has rank $d$ and is in general position. It suffices to show that $B[\alpha]=A[\alpha]-\w[\alpha]\w[\alpha]^{\top}\succ 0$ for every set $\alpha=\{i_1,\ldots,i_{d} \}\subseteq V$. Let $\beta=V\setminus \alpha$. Then
\[
\w[\alpha]=(A\mathbf{x})[\alpha]=\left(\begin{matrix}
A[\alpha]&A[\alpha, \beta]
\end{matrix}\right)\left(\begin{matrix}\mathbf{x}[\alpha]\\\mathbf{x}[\beta]\end{matrix}\right)=A[\alpha]\mathbf{x}[\alpha]+A[\alpha, \beta]\mathbf{x}[\beta].
\]
Thus, if we set $\y_{\alpha}=\mathbf{x}[\alpha]+A[\alpha]^{-1}A[\alpha, \beta]\mathbf{x}[\beta]$, then $A[\alpha]\y_{\alpha}=\w[\alpha]$. Our claim is established if we can show that
\[
\begin{pmatrix}
1&\w[\alpha]^{\top}\\
\w&A[\alpha]
\end{pmatrix} \succ 0\, ,  \: \text{or equivalently} \:
1-\w^{\top}[\alpha]A[\alpha]^{-1}\w[\alpha]=1-\y_{\alpha}^{\top}A[\alpha]\y_{\alpha}>0.
\]
We verify that the right-hand side inequality holds by writing
\begin{align*}
\y_{\alpha}^{\top}A[\alpha]\y_{\alpha}&=\left(\mathbf{x}[\alpha]+A[\alpha]^{-1}A[\alpha, \beta]\mathbf{x}[\beta] \right)^{\top} A[\alpha]\left( \mathbf{x}[\alpha]+A[\alpha]^{-1}A[\alpha, \beta]\mathbf{x}[\beta]\right)\\
&=\mathbf{x}[\alpha]^{\top}A[\alpha]\mathbf{x}[\alpha]+2\mathbf{x}[\alpha]^{\top}A[\alpha, \beta]\mathbf{x}[\beta]+\mathbf{x}[\beta]^{\top}A[\beta,\alpha]A[\alpha]^{-1}A[\alpha, \beta]\mathbf{x}[\beta]\\
&=\mathbf{x}[\alpha]^{\top}A[\alpha]\mathbf{x}[\alpha]+2\mathbf{x}[\alpha]^{\top}A[\alpha, \beta]\mathbf{x}[\beta]+\mathbf{x}[\alpha]^{\top}A[\alpha, \beta]\mathbf{x}[\beta]-\mathbf{x}[\beta]^{\top}A[\beta|\alpha]\mathbf{x}[\beta]\\
&=\left(\begin{matrix}
\mathbf{x}[\alpha]\\ \mathbf{x}[\beta]
\end{matrix}\right)^{\top}
\left(\begin{matrix}
A[\alpha]&A[\alpha, \beta]\\
A[\beta,\alpha]&A[\beta]
\end{matrix}
\right)
\left(\begin{matrix}
\mathbf{x}[\alpha]\\
\mathbf{x}[\beta]
\end{matrix}\right)-\mathbf{x}[\beta]^{\top}A[\beta|\alpha]\mathbf{x}[\beta]\\
&=\mathbf{x}^{\top}A\mathbf{x}-\mathbf{x}[\beta]^{\top}A[\beta|\alpha]\mathbf{x}[\beta]\\
&\le \mathbf{x}^{\top}A\mathbf{x}\qquad\text{(since $A\succeq 0$ and therefore $A[\beta|\alpha]\succeq 0$)}\\
&<1.
\end{align*}
Therefore, $B[\alpha]\succ 0$ and $\rk(B[\alpha])=|\alpha|=d$. This shows that $B=A-\w\w^{\top}\in \rm{S}_{+}(p,d)$ is in general position. Now by using Theorem 1.20 and the Guttman rank additivity formula in \cite{Z5} we conclude that
\[
\begin{pmatrix}
	1&\w^{\top}\\
	\w&A
\end{pmatrix}\in \rm{S}_{+}(p+1,d) \: \text{is in general position.}
\]
(b)~ First note that for a vector $\u\in \R^p$, $\rk(A+\u\u^{\top})=\rk(A)$ if and only if $\u\in\texttt{Range}(A)$, otherwise $\rk(A+\u\u^{\top})=\rk(A)+1$ \cite{HO10}. Now it suffices to find a vector $\u\in\R^p\setminus \texttt{Range}(A)$ such that $p_{\alpha}(\u)=\det\left(A[\alpha]+\u[\alpha]\u[\alpha]^{\top}\right)\ne 0$  for each $\alpha=\{i_1,\ldots, i_{d+1} \}\subseteq V$. For this, we note that $\texttt{Range}(A)$ and the zeros of the polynomials $p_{\alpha}(\x)$ are all algebraic sets $\subsetneq \R^p$ and therefore have zero measures. Thus we can choose a vector $\u$ in the nonempty set $\R^p\setminus \left(\texttt{Range}(A) \bigcup_{\alpha}\{ \x:\: p_{\alpha}(\x)=0\}\right)$. For this vector the matrix $A+\u\u^{\top}\in \rm{S}_{+}(p,d+1)$ is in general position. The rest of the proof is similar to that of Part (a).
\end{proof}
\begin{Remark}\label{rem:ext}
As the proof shows, Part (a) of Lemma \ref{lem:ext} holds in particular for $\w=A\x$, where $\x=\begin{pmatrix}
	\x_{\tau}\\ 0
\end{pmatrix}$ and $\x_{\tau}\in \R^{d}$ satisfies $\x_{\tau}^{\top} A[\tau]\x_{\tau}<1$.
\end{Remark}
\subsection{Graphical Gaussian models} Let $\G=(V, E)$ be a graph with $p$ vertices. A Gaussian distribution  $\N_{p}(0,\Sigma)$ is said to be Markov with respect to $\G$ if
\[
(\Sigma^{-1})_{i,j}=0 \:\;  \text{whenever} \:\;  ij\not\in E.
\]
The graphical Gaussian model over $\G$, denoted by $\mathscr{N}(\G)$, is the family of $\N_{p}(0,\Sigma)$ Markov with respect to $\G$. Let
\[
\mathrm{Z}_{\G}=\{A\in S(p):\: A_{i,j}=0 \: \: \text{for each}\:\: ij\notin E \} \:\; \text{and} \;\: \mathrm{P}_{\G}= \mathrm{Z}_{\G}\cap \rm{S}_{++}(p).
\]
 Then $\mathrm{Z}_{\G}$ is an $|E|$-dimensional linear space and $\mathrm{P}_{\G}$ is an open convex cone in $\mathrm{Z}_{\G}$.
Note that $\mathrm{P}_{\G}$ is in fact the set of inverse covariance matrices for the graphical Gaussian model $\mathscr{N}(\G)$.

\subsection{The maximum likelihood problem} Suppose $\bx_1, \ldots,\bx_n$ are  $n$ observations taken from $\N_p(0,\Sigma)\in\mathscr{N}(\G)$. Then the MLE of $\Sigma$ is the solution to the following optimization problem \cite{H10}:
\begin{align*}
\underset{\Sigma}{\text{argmax}}&\left( -\frac{n}{2}\log\det\Sigma-\frac{n}{2}\mathrm{tr}(\Sigma^{-1}\Sc) \right)\\
 & \text{subject to} \: \left(\Sigma^{-1}\right)_{i,j}=0, \quad \forall ij\not\in E\\
&\qquad\qquad\qquad\quad\:\,\Sigma \succ 0.
\end{align*}
 Recall that $\Sc=1/n\sum_{i=1}^{n}\bx_i\bx_i^{\top}$ is the sample covariance matrix. In terms of the inverse covariance matrix $\Omega=\Sigma^{-1}$, this optimization problem can be recast as the convex optimization problem:
\begin{align}\label{eq:convex}
\underset{\Omega}{\text{argmin}} & \left(-\frac{n}{2}\log\det \Omega+\frac{n}{2}\mathrm{tr}(\Omega \Sc) \right)\\
&\notag \text{subject to} \:  \Omega\in \mathrm{P}_{\G}.
\end{align}
\noindent A standard optimization procedure such as in \cite{H10} or \cite{D8} shows that the dual problem of \eqref{eq:convex} solves the following concave optimization problem.
\begin{align}\label{eq:opt}
\notag \underset{\Sigma}{\text{argmax}}
& \: \left( \log\det\Sigma\right) \\
 \qquad \text{subject to}& \begin{cases}
 & \Sigma_{i,j}=\Sc_{i,j},\: \:\forall ij\in E\\
 & \Sigma\succ 0,
\end{cases}
\end{align}
which has a unique solution if and only if the feasible set determined by Equation \eqref{eq:opt} is non-empty. Note that the existence of a matrix $\Sigma$ that satisfies the condition given by Equation \eqref{eq:opt} is essentially a positive definite completion problem. For the future reference we record the conclusion as follows.

\begin{Proposition}\label{prop:completion}
Suppose $\x_1,\ldots,\x_n$ are $n$ observations taken from $\N_p(0,\Sigma)\in\mathscr{N}(\G)$. Then the MLE of $\Sigma$ exists if and only if there is a matrix $P\in \rm{S}_{++}(p)$ such that $P_{i,j}=\mathcal{S}_{i,j}$ for every $ij\in E$.
\end{Proposition}



\section{ The Gaussian rank of a graph}\label{sec:grank} In this section we give another description for the Gaussian rank of a graph using its unique property with respect to the positive definite completion problem given by Equation \eqref{eq:opt}. An advantage of this alternative description is that more easily can be verified and thus used to derive the properties of the Gaussian rank as a graph parameter.

\subsection{ An alternative description of the Gaussian rank} Let $\G=(V, E)$ be a graph. Two matrices $A, B\in\R^{p\times p}$ are said to match on $\G$, denoted by $A\match{\G} B$,  if $A_{i,j}=B_{i,j}$ for each $ij\in E$. In other words, two matrices match on $\G$ if their Euclidean projection onto $\mathrm{Z}_{\G}$ are identical. The matching relation $\match{\G}$ is transitive, and invariant under scaling \cite{Buhl93}. More precisely,
\begin{eqnarray}
\label{eq:trans} A\match{\G} B &\text{and}&U\match{\G} V \implies A+U\match{\G} B+V\\
\label{eq:scale}  A\match{\G} B&\implies& D A D\match{\G} D B D, \: \text{for every diagonal matrix $D$.}
\end{eqnarray}

 Let us recall that by definition $\rb(\G)$ is the smallest number $r$ with the property that for every observations $\x_1,\ldots, \x_{r}$ if the sample covariance matrix $\mathcal{S}$ is of rank $r$ and in general position, then the MLE exists. In light of Proposition \ref{prop:completion} an alternative description of $\rb(\G)$ can be given as follows.
\begin{Proposition}\label{prop:alt}
Let $\G$ be a graph. Then $\rb(\G)$ is the smallest number $r$ with the following property:\\
$(\star)\quad\forall A\in \rm{S}_{+}(p,r)$, if $A$ is in general position $\implies \exists P\in \PD{p}$ such that $P\match{\G}A$.
\end{Proposition}
\begin{Remark}\label{rem:weak}
 Similarly, $\ra(\G)$ is the smallest number $r$ with the following property:\\
Almost surely $\forall A\in \rm{S}_{+}(p,r)$, if $A$ is in general position $\implies \exists P\in \PD{p}$ such that $P\match{\G}A$.
\end{Remark}

\subsection{ Some basic properties of the Gaussian rank}\label{subsec:basic}
In this subsection we list some basic properties of the Gaussian rank as a graph parameter.  Let us keep in mind that $\rb(\G)\le r$, for some positive integer $r$ if the condition ($\star$) in Proposition \ref{prop:alt} is satisfied. Also, because of the scale invariant property of matching relation given by Equation \eqref{eq:scale}, whenever convenient we can assume that $A\in \rm{S}_{+}(p,r)$ is a correlation matrix, that is, the diagonals of $A$ are all $1$.
\begin{Proposition}\label{prop:ge}
Let $r> \rb(\G)$. Then for every in general-position matrix $A\in \rm{S}_{+}(p,r)$ there is a matrix $P\in \rm{S}_{++}(p)$ such that $P\match{\G} A$.
\end{Proposition}
\begin{proof}
Let $A=\sum_{i=1}^{r}\w_{i}\w_i^{\top}$ be a decomposition of $A$ as in Equation \eqref{eq:decomp}. Let us set $B=\sum_{i=1}^{\rb(\G)}\w_{i}\w_i^{\top}$ and $C=\sum_{i=\rb(\G)+1}^{r}\w_{i}\w_i^{\top}$. Since $A$ is in general position the matrix $B\in \rm{S}_{+}(p, \rb(\G))$ is also in general position. Thus there is a matrix $Q\in \rm{S}_{++}(p)$ such that $Q\match{\G} B$. Equation \eqref{eq:trans} now shows
$A=B+C\match{\G} Q+C\in \rm{S}_{++}(p)$ as desired.
\end{proof}
\noindent In relation to the D\&L problem (II) Proposition \ref{prop:ge} is an important property that intuitively seems obvious. The reason is that if the MLE exists for every generic $\rb(\G)$ observations, then we expect this to be true for every generic $r\ge \rb(\G)$ observations as well.
\begin{Proposition}\label{prop:-v} Let $v\in V$ be a vertex of $\G$. Then the following holds.
\begin{enumerate}
\item[(a)] $\rb(G-v)\leq \rb(\G)\leq \rb(\G-v)+1$.
\item[(b)] $\rb(\G)=\rb(\G-v)+1$ if  $v$ is adjacent to all other vertices.
\end{enumerate}
\end{Proposition}
 \begin{proof} First, without loss of generality we assume that $p\ge 2$ and $v=1$.\\
\noindent (a)~  Let us set $r=\rb(\G)$. The claim is that $\rb(\G-v)\le r$. Since $\rb(\G-v)\le p-1$ it suffices to consider the case $r\le p-1$. Let $A\in \rm{S}_{+}(p-1,r)$ be in general position.  By Part (a) in Lemma \ref{lem:ext} there is a vector $\w\in\R^{p-1}$ such that
\[
B=
\begin{pmatrix}
	1&\w^{\top}\\
	\w&A
\end{pmatrix}\in \rm{S}_{+}(p,r) \: \text{is in general position.}
\]
Let $P\in \rm{S}_{++}(p)$ such that $P\match{\G} B$. If we set $\alpha=V(\G-v)$, then it is clear that $A\match{\G-v}P[\alpha]\in \rm{S}_{++}(p-1)$.\\

\noindent Now we show that $\rb(\G)\le \rb(\G-v)+1$. Let us set $r_0=r(\G-v)$ and let $A\in \rm{S}_{+}(p,r_0+1)$ be in general position. We assume without loss of generality that $A$ is a correlation matrix. If we partition $A$ as
\[
A=\begin{pmatrix}
	1& \u^{\top}\\
	\u& B
\end{pmatrix}, \: \text{then $C=B-\u\u^{\top}\succeq 0$.}
\]
By using the Guttman rank additivity formula in \cite{Z5} and the fact that $C[\alpha]=B[\alpha]-\u[\alpha]\u[\alpha]^{\top}$, for each $\alpha\subseteq V$, we conclude that $C\in \rm{S}_{+}(p-1,r_0)$ is in general position. Therefore, there is a matrix $Q\in \rm{S}_{++}(p-1)$ such that $Q\match{\G-v} C$. Let us set
\[
P=\begin{pmatrix}
	1&\u^{\top}\\
	\u&Q+\u\u^{\top}
\end{pmatrix}\in \rm{S}_{++}(p).
\]
The transitive property of the matching relation given by Equation \eqref{eq:trans} now shows that $Q+\u\u^{\top}\match{\G-v} B$ and therefore $P\match{\G} A$. \\

\noindent (b)~ Suppose $v=1$ is adjacent to $2,\ldots, p$. As before, let us set $r=\rb(\G)$. By Part (a) above $\rb(\G-v)\ge r-1$. Thus, it suffices to show that $\rb(\G-v)\le r-1$. For this, let $A\in \rm{S}_{+}(p-1,r-1)$. Clearly, it suffices to consider the case $r\le p-1$. By Part (b) of Lemma \ref{lem:ext} there is a vector $\u\in \R^{p-1}$ such that the matrix
\[
\begin{pmatrix}
	1&\u^{\top}\\
	\u& A+\u\u^{\top}
\end{pmatrix}\in \rm{S}_{+}(p,r) \: \text{is in general position.}
\]
Therefore, there is a matrix $P\in \rm{S}_{++}(p)$ such that
\[
P=\begin{pmatrix}
	1&\u^{\top}\\
	\u& Q
\end{pmatrix}\match{\G}
\begin{pmatrix}
	1&\u^{\top}\\
	\u& A+\u\u^{\top}
\end{pmatrix}, \: \text{for some matrix $Q\in \rm{S}_{++}(p-1)$.}
\]
Thus, $A+\u\u^{\top}\match{\G-v} Q$ and consequently $A\match{\G-v} Q-\u\u^{\top}\in \rm{S}_{++}(p-1)$.
 \end{proof}

\begin{Corollary}\label{cor:sub}
Let $\Hc$ be a subgraph of $\G$. Then $\rb(\Hc)\le \rb(\G)$.
\end{Corollary}
\begin{proof}
Every subgraph of $\G$ is obtained by removing successively a finite number of vertices and edges from $\G$. Therefore, it suffices to show that $\rb(\G-v)\le \rb(\G)$, for each vertex $v\in V$, and $\rb(\G-e)\le \rb(\G)$, for each edge $e\in E$. The latter is obvious using the condition $(\star)$ in Proposition \ref{prop:alt}.
\end{proof}

A graph $\G$ is said to be the clique sum of two subgraphs $\G_1=(V_1,E_1)$ and $\G_2=(V_2, E_2)$ if  $V=V_1\cup V_2$, $E=E_1\cup E_2$ and $G[V_1\cap V_2]$ is a complete graph (including the empty graph). We write this as $
\G=\G_1\underset{V_1\cap V_2}{\oplus} \G_2$. The following proposition is now immediate using a standard completion process given by \cite{G84} or \cite{L96}.
\begin{Proposition}\label{prop:2}
Suppose $\G=\G_1\underset{V_1\cap V_2}{\oplus} \G_2$. Then $\rb(\G)=\max\{\rb(\G_1),\rb(\G_2)\}$. In particular, if  $\G_1,\ldots,\G_k$ are all the connected components of $\G$, then
\[
\rb(\G)=\max\left\{\rb(\G_i): i=1,\ldots,k\right\}.
\]
\end{Proposition}

\begin{Remark}
 In relation to the D$\&$L problem (I) we can show (by slightly modifying our proofs and using the alternative description of $\ra(\G)$ given by Remark \ref{rem:weak}) that the properties of the Gaussian rank discussed in this subsection are also valid properties of the weak Gaussian rank. This fact however will not be used in the paper.
\end{Remark}
\section{The proof of Theorem \ref{thm:main}}\label{sec:main} In this section we prove Theorem \ref{thm:main}, that is, the bounds given by Equation \eqref{eq:main}. The lower and the upper bounds for $\rb(\G)$ are proved in separate subsections.

\subsection{ The upper bound: $\rb(\G)\le \delta^*(\G)+1$} First we state and prove the following key lemma.
\begin{Lemma}\label{lem:key}
Let $v\in V$ be a vertex of $\G$. If $\rb(\G-v)\ge \textrm{deg}_{\G}(v)+1$, then $\rb(\G)=\rb(\G-v)$.
\end{Lemma}
\begin{proof}
If $\deg_{\G}=0$, then by Proposition \ref{prop:2} $\rb(\G)=\max\{1,\rb(\G-v)=\rb(\G-v)$. Thus we assume that $\deg_{\G}(v)\ge 1$. For convenience, let us assume that $v=1$ and $\mathrm{ne}(v)=\{2,\ldots, \deg_{\G}(v)+1\}$. Note that $V(\G-v)=\{2,\ldots,p\}$. Now we set 
\[
r_0=\rb(\G-v)\,, \:\: \tau=\{2,\ldots,r_0\} \:\; \text{and} \;\; \gamma=V(\G-v)\setminus \tau
\]
Note that $\mathrm{ne}(v)\subseteq \tau$ since $r_0\ge \deg_{\G}(v)+1$. Let $A\in \rm{S}_{+}(p,r_0)$. We need to show that there is a matrix $P\in \rm{S}_{++}(p)$ such that $P\match{\G}A$. For this, we assume that $A$ is a correlation matrix and therefore can be partitioned as
\[
A=
\begin{pmatrix}
1&\u^{\top}\\
\u&B
\end{pmatrix}=
\begin{pmatrix}
	1&\u[\tau]^{\top}&\u[\gamma]^{\top}\\
	\u[\tau]&A[\tau]&A[\tau,\gamma]\\
	\u[\gamma]&A[\gamma,\tau]&A[\gamma]
\end{pmatrix}
, \: \text{where} \: \; B=
\begin{pmatrix}
	A[\tau]&A[\tau,\gamma]\\
	A[\gamma,\tau]&A[\gamma]
\end{pmatrix}.
\]
Note also that $\u\in \R^{p-1}$ and $B\in \rm{S}_{+}(p-1,r_0)$ is in general position.  Let us set $\x_{\tau}=A[\tau]^{-1}\u[\tau]$. Now we have
\[
\begin{pmatrix}
	1&\u[\tau]^{\top}\\
	\u[\tau]&A[\tau]
\end{pmatrix}\succ 0 \;\, \text{and therefore} \;\, 1-\u[\tau]^{\top}A[\tau]^{-1}\u[\tau]>0.
\]
This implies that $\x_{\tau}^{\top}A[\tau]\x_{\tau}=\x_{\tau}^{\top}B[\tau]\x_{\tau}<1$. By Part (a) of Lemma \ref{lem:ext} and  Remark \ref{rem:ext}, if we set $\w=B\x$, then $B-\w\w^{\top}\in \rm{S}_{+}(p-1,r_0)$ is in general position. Therefore, there is a matrix $Q\in \rm{S}_{++}(p-1)$ such that $Q\match{\G-v} B-\w\w^{\top}$. Let us set
\[
P=\begin{pmatrix}
	1&\w^{\top}\\
	\w& Q+\w\w^{\top}
\end{pmatrix}\in \rm{S}_{++}(p).
\]
By Equation \eqref{eq:trans} it is clear that $Q+\w\w^{\top}\match{\G-v} B$. Also, since $\w[\tau]=A[\tau]\x_{\tau}=\u[\tau]$ and $\mathrm{ne}(v)\subseteq \tau$ we have $P_{1j}=A_{1j}$ whenever $j\in \mathrm{ne}(v)$. Therefore, $P\match{\G}A$.
\end{proof}
   We now prove that the upper bound $\rb(\G)\leq \delta^*(\G)+1$ holds.
\begin{proof} By mathematical induction let us assume that for any graph $\Hc$ with fewer than $p$ vertices $\rb(\Hc)\le \delta^*(\mathcal{H})+1$. Now let $\G$ be a graph with $p$ vertices. We assume without loss of generality that $p\ge2$. Let $v\in V$ such that $\textrm{deg}_{\G}(v)=\delta(G)$. On the contrary, let us assume that 
$\rb(G)\ge \delta^*(\G)+2$, then by Part (a) of Proposition \ref{prop:-v} we obtain
\[ \rb(\G-v)\ge \rb(G)-1\ge \delta^*(\G)+1\ge \textrm{deg}_{\G}(v)+1.
\]
 Lemma \ref{lem:key} now implies that $\rb(\G)=\rb(\G-v)$. The induction hypothesis then implies that $\rb(\G)=\rb(\G-v)\le \delta^*(\G-v)+1$. The fact that $\delta^*(\G)=\max\{\delta(\G), \delta^*(\G-v)\}$ \cite{M78} implies that $\rb(\G)\le \delta^*(\G)+1$. 
\end{proof}

\begin{Remark}
Note that Lemma \ref{lem:key} also holds for the weak Gaussian rank, that is,  for any graph $\G$ if $\ra(\G-v)\ge \deg_\G+1$, then $\ra(\G)=\ra(\G-v)$. Consequently, this implies $\ra(\G)\le \delta^{*}(\G)+1$. But the latter more easily follows from the fact that $\ra(\G)\le \rb(\G)\le \delta^{*}(\G)+1$.
\end{Remark}

\subsection{The lower bound: $\kappa^*(\G)+1\le \rb(\G)$} To prove that this lower bound holds we mainly rely on a theorem in \cite{L89}. This theorem associates the connectivity of a graph with its certain geometric representations. More details are as follows. Let $\G=(V,E)$ be graph. An orthonormal representation of $\G$ in $\R^d$ is a function $\phi:V\to \R^d$ assigning a unit vector $\u_i$ to each vertex $i\in V$ such that $\u_i^{\top} \u_j=0 $ whenever $ij\not\in E$.
In words, $\phi$ assigns to each vertex a unit vector in $\R^d$ such that the vectors assigned to nonadjacent vertices are orthogonal. Now a general-position orthonormal representation of $\G$ in $\R^d$ is an orthonormal representation $\phi$ in $\R^d$ such that the $p$ assigned vectors $\phi(1)=\u_1,\ldots, \phi(p)=\u_p$ are in general position in $\R^d$. The next theorem is crucial for proving the lower bound in Theorem \ref{thm:main}.
\begin{Theorem}[Lovasz et al. \cite{L89, L00}]\label{thm:Lovasz} If $\G=(V,E)$ is a graph with $p$ vertices, then the following are equivalent:
\begin{enumerate}
\item[\rm{(i)}] $\G$ is $(p -d )$-connected;
\item[\rm{(ii)}] $\G$ has a general-position orthonormal representation in $\R^d$.
\end{enumerate}
\end{Theorem}
We now proceed to prove that the lower bound $\kappa^{*}(\G)+1 \le \rb(\G)$ holds.
 \begin{proof} Since $\rb(\Hc)\le\rb(\G)$ for every subgraph $\Hc$ of $\G$, in light of Part (a) and Part (b) in Remark \ref{rem:star}, it suffices to show that for every graph $\G$ we have $\rb(\G)\ge \kappa(\G)+1$. First we set $k=\kappa(\G)$ and $d=p-k$. By Theorem \ref{thm:Lovasz} there are in general-position unit vectors $\u_1,\ldots,\u_p\in \R^d$ such that $\u_i^{\top}\u_j =0$ for each $ij\not\in E$. Let $\Omega=(\u_i^{\top}\u_j)_{1\le i,j\le p}$. By Lemma \ref{lem:equiv} we have $\Omega\in \rm{S}_{+}(p,d)$ is in general position. Let $\w_1,\ldots, \w_{k}\in \R^p$ be a basis of $\texttt{Null}(\Omega)$ and set
\[
A=
\left(
\begin{matrix}
\w_{1}&
\ldots&
\w_{k}
\end{matrix}
\right)
\left(
\begin{matrix}
\w_{1}^{\top}\\
\vdots\\
\w_{k}^{\top}
\end{matrix}
\right)\in \rm{S}_{+}(p,k).
\]
Corollary \ref{cor:cor} in Appendix B now implies that the matrix $A\in  \rm{S}_{+}(p,k)$ is in general position. Now let $P\in\rm{S}(p)$ such that $P\match{\G}A$. Then 
\[
\tr(P\Omega)=\sum_{ij\in E}P_{i,j}\Omega_{i,j}=\sum_{ij\in E}A_{i,j}\Omega_{i,j}=\tr(A \Omega)=0\quad\text{(note that $A\Omega=0$).}
\]
This shows that $P\notin \PD{p} $, since $\PD{p}$ is a self dual cone, that is,
\[
\rm{S}_{++}(p)=\{Q\in\rm{S}(p): \: \tr(Q B)>0 \:\: \text{for every} \: B\in  \rm{S}_{+}(p)\setminus\{ 0\} \},
\]
and $\Omega\in \rm{S}_{+}(p)\setminus \{0\}$. Thus there is no matrix $P\in\rm{S}_{++}(p)$ such that $P\match{\G} A$. This shows that $\rb(\G)\geq \kappa(\G)+1$.
\end{proof}

\section{Some applications of  Theorem \ref{thm:main}}\label{sec:app}
A useful application of the bounds given by Theorem \ref{thm:main} is that when the lower and upper bounds are equal the guassian rank is exactly determined. In this section we briefly discuss for which graphs $\kappa^{*}(\G)=\delta^*(\G)$. We also use Theorem \ref{thm:main} to obtain a sharp numerical upper bound for the guassian ranks of the so-called planar graphs. 
\subsection{Gaussian ranks of symmetric graphs and random graphs}
Let $\G=(V,E)$ be a graph. A permutation $\sigma: V\to V$ is said to be an automorphism of $\G$ if $ij\in E\implies \sigma(i)\sigma(j)\in E$. The graph $\G$ is said to be symmetric if for any two edges $ij$ and $i'j'\in E$ there is an automorphism $\sigma$ of $\G$ such that $\sigma(i)=i'$ and $\sigma(j)=j'$ (please see Chapter 27 in \cite{Gr95} for a detailed discussion of symmetric graphs). For example, one can check that the graphs given by Figure \ref{fig:aa} and Figure \ref{fig:bb} are symmetric. The symmetric property in particular implies that $\G$ is regular, that is, there is a positive integer $k$ such that $\deg_{\G}(v)=k$ for every vertex $v\in V$. An interesting feature of a symmetric graph $\G$ is that $\kappa(\G)=\delta(\G)$. This consequently implies that $\kappa^*(\G)=\delta^*(\G)$. Theorem \ref{thm:main} therefore implies that for a symmetric graph the guassian rank is exactly determined as
 \begin{equation}\label{eq:equality}
\rb(\G)=\kappa^*(\G)=\delta^*(\G).
\end{equation}
 Note that Equation \ref{eq:equality} can hold for many non-symmetric graphs as well, such as the regular graph given by Figure \ref{fig:cc} or even non-regular graphs such as grids. To the best our knowledge the class of all graphs $\G$ satisfying $\kappa^{*}(\G)=\delta^{*}(\G)$, or even $\kappa(\G)=\delta(\G)$, is not fully characterized. In the context of random graphs these graph parameters are identical almost surely for all random graphs. To be precise, a random graph $\G(\epsilon)$ is a graph in which the edges are selected by a sequence of (independent) Bernoulli trials with probability $0< \epsilon<1$ \cite{Gr95}.

\begin{Theorem}[Bollob\'{a}s et al.]
Almost surely for every random graph $\G(\epsilon)$ we have $\kappa(\G(\epsilon))=\delta(\G(\epsilon))$.
\end{Theorem}
 \noindent The next result now follows from this theorem and Theorem \ref{thm:main}.
\begin{Corollary}
Almost surely for every random graph $\G(\epsilon)$ we have
\[\rb(\G(\epsilon))=\kappa^*(\G(\epsilon))=\delta^*(\G(\epsilon)).
\]
\end{Corollary}

\begin{figure}[htbp]
\centering
\subfigure[]{
\includegraphics[scale=.06]{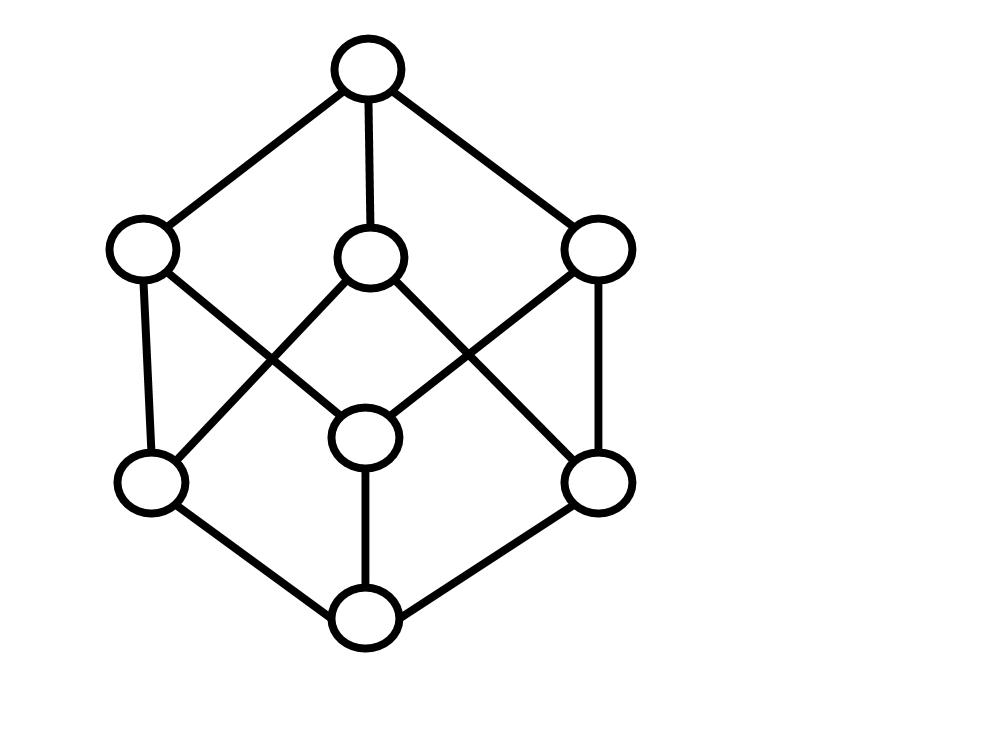}
\label{fig:aa}
}
\hspace{2.cm}
\subfigure[]{
\includegraphics[scale=.07]{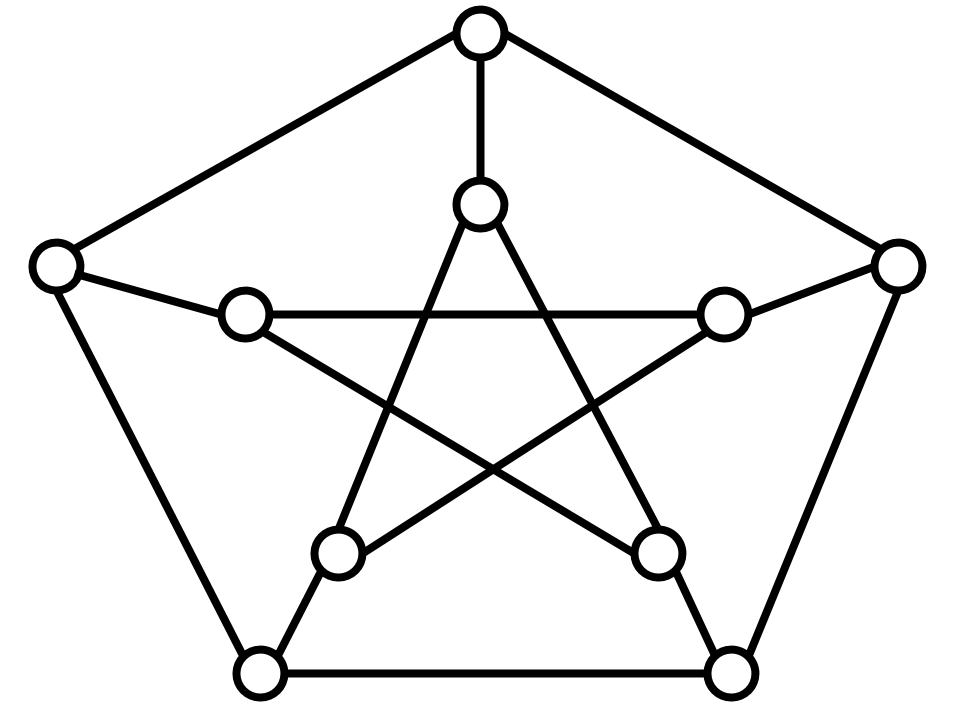}
\label{fig:bb}
}
\hspace{2.cm}
\subfigure[]{
\includegraphics[scale=.06]{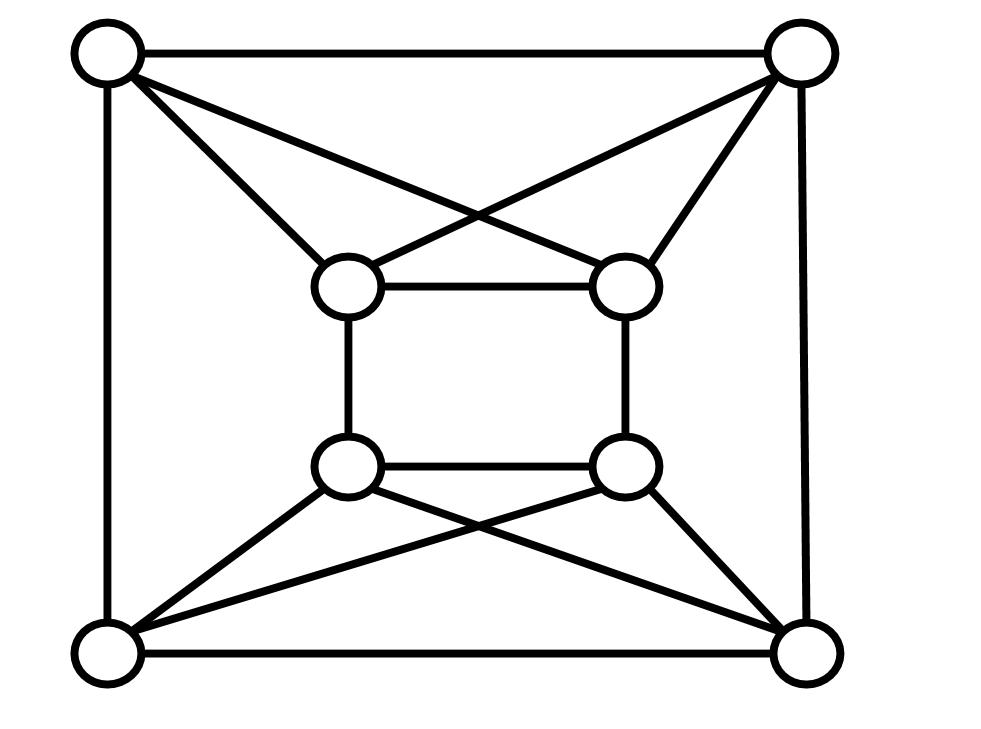}
\label{fig:cc}
}
\caption{Denoted graphs in (a) and (b) are symmetric. The graph given by  (c) is not symmetric. In each graph the equality in Equation \eqref{eq:equality} holds and the Gaussian ranks are easily computed to be  $4$, $4$ and $5$, respectively.}
\end{figure}
\subsection{On the Gaussian ranks of planar graphs} A planar graph is a graph that can be drawn in a plane without the edges crossing each other. In graph theory it is well-known that planar graphs are at most $5$-degenerate, that is, $\delta^*(\G)\le 5$ if $\G$ is planar. Theorem \ref{thm:main} therefore implies that when $\G$ is a planar graph $\rb(\G)\le 6$. Note that the upper bound $6$ is tight. For this, consider the planar graph $\G$ in  Figure \ref{fig:deg}. Since $\kappa^*(\G)=\delta^*(\G)=5$ by Equation \eqref{eq:equality} we have $\rb(\G)=6$.
\begin{figure}[htbp]
	\centering
		\includegraphics[scale=.08]{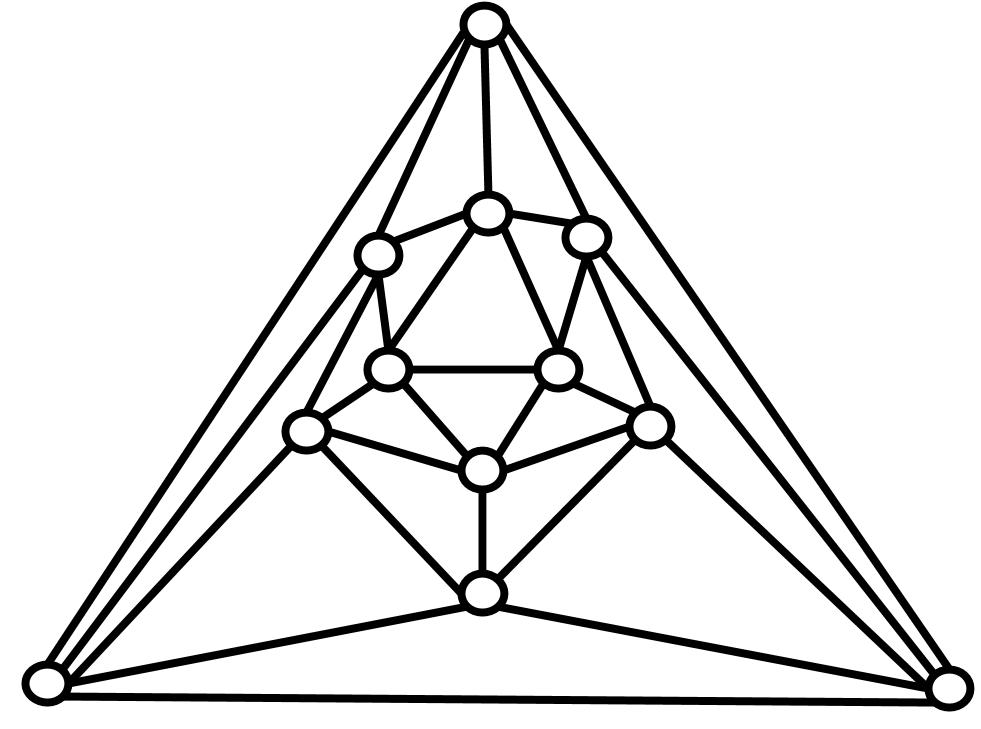}
	\caption{A planar graph with degeneracy number $5$.}
	\label{fig:deg}
\end{figure}


\section*{Appendix A}
Let $\G=(V,E)$ be a graph.
\begin{itemize}
 \item[(a)] A tree-decomposition of $\G$ is a pair $D=( S, T)$  with  $S=\{X_i :\: i \in I\}$  a collection of subsets of $V$ and  $T = (I, F)$  a tree, with one vertex for each subset of $S$, such that the following three conditions are satisfied:
\begin{enumerate}
\item  $\bigcup_{i\in I} X_i= V$,
\item  for all edges $(v,w)\in E$ there is a subset $X_i \in S$ such that both $v$ and $w$ are contained in $X_i$,
\item  for each vertex $u$ the set of vertices  $\{i : \: u\in X_i\} $ forms a sub tree of $T$.
\end{enumerate}
\item[(b)]The width of the tree-decomposition $\left(S, T\right)$ is $\mathrm{tw}(S,T)= \max\{|X_i|- 1: \: i\in I\}$.
\item[(c)] The treewidth of $\G$ is $\mathrm{tw}(\G)= \min\left\{ \mathrm{tw}(S,T): \: (S,T) \:\text{is a tree-decomposition of  $\G$}\right\}
$.
\end{itemize}

\section*{Appendix B}
\begin{Lemma}\label{lem:tr}
Suppose $A, B\in \rm{S}_{+}(p)$. Then $AB=0$ if and only if $\tr(AB)=0$.
\end{Lemma}
\begin{proof}
This easily follows from the fact that $\tr(AB)=\tr\left(A^{1/2}B^{1/2}A^{1/2}B^{1/2}\right)=\|A^{1/2}B^{1/2}\|_{F}$.
\end{proof}

\begin{Proposition}\label{prop:genker}
Let $A\in \rm{S}_{+}(p,d)$. Suppose $A[\alpha]$ is non-singular, for some $\alpha=\{i_1,\ldots, i_d\}\subseteq [p]$. Let us set $k=p-d$ and $\beta=V\setminus \alpha$. Then,
\[
\text{if  $B\in \rm{S}_{+}(p,k)$ is such that $AB=0 \implies B[\beta]$ is non-singular.}
\]
\end{Proposition}
\begin{proof}
First note that $A[\alpha]$ is non-singular, thus by the Guttman rank additivity formula we have $\rk(A)=\rk(A[\alpha])+\rk(A[\beta| \alpha])$. Therefore, $A[\beta| \alpha]=0$ and $A$ can be partitioned as
\[
A=\begin{pmatrix}
	I_d&0\\
	A[\beta,\alpha]A[\alpha]^{-1}&I_k
\end{pmatrix}
\begin{pmatrix}
	A[\alpha]&0\\
	0&0
\end{pmatrix}
\begin{pmatrix}
	I_d&A[\alpha]^{-1}A[\alpha,\beta]\\
	0&I_k
\end{pmatrix},
\]
where $I_d$ and $I_k$ are identity matrices in $\R^{d\times d}$ and $\R^{k\times k}$. Note now that
\begin{align*}
\tr(AB)&=\tr\left(\begin{pmatrix}
	I_d&0\\
	A[\beta,\alpha]A[\alpha]^{-1}&I_k
\end{pmatrix}
\begin{pmatrix}
	A[\alpha]&0\\
	0&0
\end{pmatrix}
\begin{pmatrix}
	I_d&A[\alpha]^{-1}A[\alpha,\beta]\\
	0&I_k
\end{pmatrix}B  \right)\\
&=\tr\left(
\begin{pmatrix}
	A[\alpha]&0\\
	0&0
\end{pmatrix}
\begin{pmatrix}
	I_d&A[\alpha]^{-1}A[\alpha,\beta]\\
	0&I_b
\end{pmatrix}B \begin{pmatrix}
	I_d&0\\
	A[\beta,\alpha]A[\alpha]^{-1}&I_k
\end{pmatrix}
 \right)\\
&=\tr\left(\begin{pmatrix}
	A[\alpha]&0\\
	0&0
\end{pmatrix}
\begin{pmatrix}
	C[\alpha]&C[\alpha,\beta]\\
	C[\beta,\alpha]&C[\beta]
\end{pmatrix}
\right)=0,
\end{align*}
where
\[
C=\begin{pmatrix}
	C[\alpha]&C[\alpha,\beta]\\
	C[\beta,\alpha]&C[\beta]
\end{pmatrix}=
\begin{pmatrix}
	I_d&A[\alpha]^{-1}A[\alpha,\beta]\\
	0&I_b
\end{pmatrix}B \begin{pmatrix}
	I[\alpha]&0\\
	A[\beta,\alpha]A[\alpha]^{-1}&I_k
\end{pmatrix}\in \rm{S}_{+}(p).
\]
By Lemma \ref{lem:tr} this implies that
\[
\begin{pmatrix}
	A[\alpha]&0\\
	0&0
\end{pmatrix}
\begin{pmatrix}
	C[\alpha]&C[\alpha,\beta]\\
	C[\beta,\alpha]&C[\beta]
\end{pmatrix}=
\begin{pmatrix}
	A[\alpha]C[\alpha]&A[\alpha]C[\alpha,\beta]\\
	0&0
\end{pmatrix}=
\begin{pmatrix}
	0&0\\
	0&0
\end{pmatrix}.
\]
Therefore, $C[\alpha]=0$, $C[\alpha,\beta]=0$ and $C[\beta,\alpha]=C[\alpha,\beta]^{\top}=0$. Thus
\[
k=\rk(B)=\rk(C)=\rk(C[\beta])=\rk(B[\beta]). \qquad \text{(Note that  $C[\beta]=B[\beta]$.)}
\]
This shows that $B[\beta]$ is non-singular.
\end{proof}
\begin{Corollary}\label{cor:cor}
Suppose $A\in \rm{S}_{+}(p,d)$ is in general position. If $B\in \rm{S}_{+}(p, p-d)$ such that $AB=0$, then $B$ is also in general position.
\end{Corollary}
\begin{proof}
Set $k=p-d$ and let $\beta=\{i_1,\ldots,i_k\}\subseteq[p]$. If we set $\alpha=[p]\setminus \beta$, then $A[\alpha]$ is non-singular. Therefore, by Proposition \ref{prop:genker} $B[\beta]$ is non-singular.
\end{proof}

\begin{thebibliography}{9}

%
%
%
%
\bibitem{Bo85}
\textsc{Bollob\'{a}s, B. \and Thomason, A.}
Random graphs of small orders.
\textit{Random Graphs ``83 Based on lectures presented at the 1st Pozn\'{a}n Seminar on Random Graphs"}.
North-Holland, 1985.


\bibitem{Buhl93}
\textsc{Buhl S. L.}
On the existence of maximum likelihood estimators for graphical Gaussian models.
\textit{Scandinavian Journal of Statistics,}
\textbf{20} (1993) 263--270.


\bibitem{D8}
\textsc{Dahl, Joachim and Vandenberghe, Lieven and Roychowdhury, Vwani}
Covariance selection for nonchordal graphs via chordal embedding.
\textit{Optimization Methods and Software.}
\textbf{23} (2008) {501--520}.

\bibitem{D72}
\textsc{Dempster, A. P.}
Covariance selection.
\textit{ Biometrics}
\textbf{28} (1972) 157--175.

\bibitem{Diestel10}
\textsc{Diestel, R.}
\textit{Graph theory: Graduate Texts in Mathematics.}
Springer--Verlag, Heidelberg, 2010.

\bibitem{Gr95}
\textsc{Graham, R. L., Gr\"{o}tschel, M. and Lov\'{a}sz, L.}
 \textit{Handbook of Combinatorics II}
 MIT Press, Cambridge, MA, USA, 1995.

\bibitem{G84}
\textsc{Grone R., Johnson C. R., Sa E. M. and Wolkowicz H.}
 Positive  definite completions of partial Hermitian matrices.
\textit{Linear  Algebra and Its Applications.}
\textbf{58} (1984) 109--124.


\bibitem{H10}
\textsc{Hastie, Trevor, Tibshirani, Robert \and Friedman, Jerome.}
 \textit{The elements of statistical learning. Data mining, inference, and prediction.}
Springer Series in Statistics. Springer-Verlag, New York, 2001.

\bibitem{HO10}
\textsc{Householder, Alston.}
{\textit The theory of matrices in numerical analysis.}
Blaisdell, New York, 196 Springer-Verlag, 2010.

\bibitem{K94}
\textsc{Kloks, Ton.}
\textit{Treewidth. Computations and Approximations, volume 842 of Lecture Notes in Computer Science.}
Springer Verlag, Berlin, 1994.


\bibitem{L96}
\textsc{ Lauritzen, Steffen.}
 \textit{Graphical models.}
 Oxford University Press, 1996.

\bibitem{L89}
 \textsc{Lov\'asz, L., Saks, M. , \and Schrijver, A.}
 Orthogonal representations and connectivity of graphs.
\textit{Linear Algebra Applications.}
 \textbf{114-115} (1989) 439–-454.

\bibitem{L00}
 \textsc{Lov\'asz, L., Saks, M. , \and Schrijver, A.}
A correction: orthogonal representations and connectivity of graphs.
\textit{Linear Algebra Applications.}
\textbf{313} (2000) 101–-105.

\bibitem{M78}
\textsc{Matula, David W.}
"Subgraph connectivity numbers of a graph" in Theory and Applications of Graphs.
\textit{Lecture Notes in Mathematics.}
 Springer--Verlag, Berlin Heidelberg New York.
\textbf{642} (1978) 371--383.

\bibitem{S86}
\textsc{Speed, T.P., \and Kiiveri, H.}
Gaussian Markov distribution over fnite graphs.
 \emph{Annals of Statistics.}
\textbf{14} (1986) 138–-150.

\bibitem{Uh12}
\textsc{Uhler, Caroline.}
Geometry of maximum likelihood estimation in Gaussian graphical models.
\textit{The Annals of Statistics.}
\textbf{40} (2012) 238--261.

\bibitem{Z5}
\textsc{Zhang, Fuzhen.}
\textit{The Schur Complement and Its Applications: Numerical Methods and Algorithms.}
Springer, Dordrecht, 2005.

\end{thebibliography}
\end{document}